\documentclass{amsart}

\usepackage{graphicx,color}

\makeatletter
\@addtoreset{equation}{section}

\makeatother

\newtheorem*{theorem*}{Theorem}
\newtheorem{theorem}{Theorem}[section]
\newtheorem{proposition}[theorem]{Proposition}
\newtheorem{corollary}[theorem]{Corollary}
\newtheorem{lemma}[theorem]{Lemma}
\theoremstyle{definition}
\newtheorem{definition}[theorem]{Definition}
\newtheorem{remark}[theorem]{Remark}

\begin{document}

\title[A characterization of cut locus for $C^1$ hypersurfaces]{A characterization of cut locus for $C^1$ hypersurfaces}
\author[Tatsuya Miura]{Tatsuya Miura}
\address{%
Graduate School of Mathematical Sciences\\
University of Tokyo\\
3-8-1 Komaba, Meguro, Tokyo\\
153-8914 Japan}
\email{miura@ms.u-tokyo.ac.jp}

\begin{abstract}
	Let $\Omega$ be an open set in $\mathbb{R}^n$ with $C^1$-boundary and $\Sigma$ be the skeleton of $\Omega$, which consists of points where the distance function to $\partial\Omega$ is not differentiable.
	This paper characterizes the cut locus (ridge) $\overline{\Sigma}$, which is the closure of the skeleton, by introducing a generalized radius of curvature and its lower semicontinuous envelope.
	As an application we give a sufficient condition for vanishing of the Lebesgue measure of $\overline{\Sigma}$.
\end{abstract}

\maketitle

\section{Introduction}

Let $n\geq2$ and $\Omega\subset\mathbb{R}^n$ be an open set with nonempty boundary $\partial\Omega$.
The {\it distance function} $d:\Omega\to(0,\infty)$ and the {\it metric projection} $\pi:\Omega\to\mathcal{P}(\partial\Omega)$ are defined by
$$d(x):=\inf_{\xi\in\partial\Omega}|\xi-x|,\quad \pi(x):=\{\xi\in\partial\Omega \mid d(x)=|\xi-x|\}.$$
The relation between the differentiability of the distance function and the number of elements of the metric projection is well-known:
\begin{theorem}[{e.g.\ \cite[Corollary 3.4.5]{CaSi04}}] \label{knownthm1}
	The distance function $d$ is differentiable at $x\in\Omega$ if and only if $\pi(x)$ is a singleton.
\end{theorem}
\noindent
Thus we define the singular set $\Sigma\subset\Omega$ called {\it skeleton} (or {\it medial axis}) by
$$\Sigma:=\{x\in\Omega \mid \pi(x)\text{ is not a singleton}\}.$$
Besides the above characterization, several properties of the skeletons are known for general open sets $\Omega$: the skeletons are $C^2$-rectifiable \cite{Al94}, in particular they have null Lebesgue measure, and they have the same homotopy type as $\Omega$ at least in the bounded case \cite{Li03} (cf.\ \cite{AlCaNgSi13}).
Skeletons (medial axes) are also studied in view of image processing, see e.g.\ \cite{AtBoEd09} and references therein.
The generality of $\Omega$ is important to contain noisy cases.

This paper studies the set $\overline{\Sigma}$, which is the closure of the skeleton in $\Omega$, called {\it cut locus} (or {\it ridge}).
The complement of the cut locus concerns the differentiability of the distance function not only at points but also in neighborhoods of points.

Cut loci have also been studied in several (generalized) settings, see e.g.\ \cite{CrMa07,ItTa01,LiNi05,MaMe03}.
The complement set $\Omega\setminus\overline{\Sigma}$ is the largest open set where the distance function is of class $C^1$.
For general $\Omega$ the distance function is a unique nonnegative viscosity solution to the simplest Eikonal equation $|\nabla u|=1$ with the zero Dirichlet condition (see e.g.\ \cite{HaNt15,Li82}).
Since this is a first order equation, it is natural to know where the solution is of class $C^1$.
We mention that cut loci also appear in the studies of other partial differential equations, see e.g.\ \cite{CrFr15,CrFr16,LiNi98}.

Unlike skeletons, some properties of cut loci crucially depend on the regularity of boundary.
In particular, it is critical whether the boundary is of class $C^2$.
If the boundary is at least $C^2$ then the cut locus behaves rather well \cite{CrMa07,ItTa01,LiNi05,MaMe03}; in particular it still has null Lebesgue measure \cite{CrMa07}.
On the other hand, in \cite[\S3]{MaMe03}, Mantegazza-Mennucci give a pathological example of a planar convex $C^{1,1}$-domain such that the cut locus has positive Lebesgue measure.

Our purpose is to find a general theory for cut loci without the $C^2$ assumption, that is, with pathological cases.
We emphasize that the theories in the above cited papers basically work for only regular sets at least $C^2$.
In this paper, as a first step, we characterize the cut locus by a geometric quantity of the boundary for a general open set with $C^1$-boundary.

We first recall a characterization of cut loci by radius of curvature in the $C^2$ case:
\begin{theorem}[{e.g.\ \cite{CrMa07}}] \label{knownthm2}
	Let $\partial\Omega$ be of class $C^2$.
	Then $x\in\Omega\setminus\overline{\Sigma}$ if and only if $\pi(x)$ is a singleton and $d(x)<\rho(\xi)$, where $\pi(x)=\{\xi\}$. 
\end{theorem}
\noindent
Here $\rho(\xi)$ is the classical inner radius of curvature of $\partial\Omega$ at $\xi$ (see Remark \ref{remradcur} for the definition).
Theorem \ref{knownthm2} is well-known: the cited paper \cite{CrMa07} proves it in terms of principal curvature $\kappa$ instead of $\rho$ (for the Minkowski distance).
Theorem \ref{knownthm2} means that the differentiability of $d$ near a point depends on not only the global shape of $\partial\Omega$ but also the local shape as curvature. 
That curvature appears in the statement also tells the importance of the $C^2$ assumption.

The main result of this paper is to characterize cut loci in the $C^1$ case by a kind of radius of curvature.
To state it we should extend the definition of radius of curvature $\rho$.
This can be easily achieved for general open sets by using locally inner touching spheres (see Definition \ref{defradcur}). 
Unfortunately, by just this extension, the ``if'' part of the above characterization is not valid even in the $C^{1,1}$ case due to the loss of rigidity of $C^2$.
Indeed, the example provided in \cite[\S3]{MaMe03} is also a counterexample to this case.
Nevertheless, we can characterize cut loci by taking a lower semicontinuous envelope of radius of curvature $\rho_*$ as Definition \ref{defradcur}.
The main result of this paper is the following:
\begin{theorem}\label{thm1}
	Let $\partial\Omega$ be of class $C^1$.
	Then $x\in\Omega\setminus\overline{\Sigma}$ if and only if $\pi(x)$ is a singleton and $d(x)<\rho_*(\xi)$, where $\pi(x)=\{\xi\}$.
\end{theorem}
\noindent
Since $\rho_*$ coincides with classical $\rho$ in the $C^2$ case, Theorem \ref{thm1} is a generalization of Theorem \ref{knownthm2}.

As an application of this characterization, in \S 4, we show that if $\partial\Omega$ is $C^{1,1}$ and almost $C^2$ then the Lebesgue measure of the cut locus vanishes.

We finally emphasize that our proof of Theorem \ref{thm1} quite differs from that of $C^2$ case; it does not calculate second derivatives and is more geometric.
The ``if'' part is proved by considering its contrapositive.
The inequality $d(x)\geq\rho_*(\xi)$ is obtained for $x\in\overline{\Sigma}\setminus\Sigma$ by seeking a sequence of suitably ``curving'' points in $\partial\Omega$ converging to $\xi$.
To state ``curving'' we use comparisons of functions.
Our proof of this part depends on the local $C^1$-graph representation of $\partial\Omega$.
The ``only if'' part is proved by showing that, for $x\in\Omega\setminus\overline{\Sigma}$, any point near $\xi$ has a locally inner touching sphere with suitably large radius.
We then use the homotopy theory as mapping degree.
This part is proved for a wider class of $\Omega$ (see Definition \ref{defnonsprper}).
Another more conceptual and geometric proof is also given.

The organization of this paper is as follows.
Some notations and known results are prepared in \S2.
Theorem \ref{thm1} is proved in \S3. 
The proof is separated into the ``if'' part \S3.1 and the ``only if'' part \S3.2.
Using Theorem \ref{thm1}, a sufficient condition for vanishing of the Lebesgue measures of cut loci is given in \S4.
Finally, we mention remarks for non-smooth cases in \S5.

\section{Envelope of inner radius of curvature}

In this section, we prepare notation and review a known result.

We first prepare some notations.
Let $B^m_r(x)$ denote an $m$-dimensional open ball of radius $r$ centered at $x\in\mathbb{R}^m$ and $\overline{B}^m_r(x)$ denote the closure.
Let $S^{m-1}_r(x)$ denote an $(m-1)$-dimensional sphere of radius $r$ centered at $x\in\mathbb{R}^m$, that is, $S^{m-1}_r(x)=\partial B^m_r(x)$.
Let $\mathcal{U}^m_x$ denote the set of open neighborhoods of $x$ in $\mathbb{R}^m$.

Then we define a generalized inner radius of curvature as follows.

\begin{definition}\label{defradcur}
	Let $\Omega\subset\mathbb{R}^n$ be an open set with nonempty boundary and $\xi\in\partial\Omega$.
	We say that an open ball $B^n_r(x)\subset\mathbb{R}^n$ is a {\it locally inner touching ball} at $\xi$ if $\xi\in\partial B^n_r(x)$ and there exists a neighborhood $U\in\mathcal{U}^n_\xi$ such that $B^n_r(x)\cap U$ is contained in $\Omega\cap U$.
	We denote the set of locally inner touching balls at $\xi$ by $\mathfrak{B}^n_\xi$.
	Then we define the {\it inner radius of curvature} at $\xi$ by
	\begin{align*}
		\rho(\xi):=
		\begin{cases}
			\sup\left\{r>0\left|B^n_r(x)\in\mathfrak{B}^n_\xi\right\}\right. & \text{if}\ \mathfrak{B}^n_\xi\neq\emptyset,\\
			0 & \text{if}\ \mathfrak{B}^n_\xi=\emptyset.
		\end{cases}
	\end{align*}
	Note that $\rho:\partial\Omega\to[0,\infty]$.
	Moreover, we denote by $\rho_*$ the lower semicontinuous envelope of $\rho$, that is, for $\xi\in\partial\Omega$
	\begin{align*}
		\rho_*(\xi):=\lim_{r\downarrow0}\inf\{\rho(\eta) \mid \eta\in B^n_r(\xi)\cap\partial\Omega\}.
	\end{align*}
	The function $\rho_*:\partial\Omega\to[0,\infty]$ is lower semicontinuous.
\end{definition}

\begin{remark}\label{remradcur}
	If $\partial\Omega$ is of class $C^2$ then both $\rho$ and $\rho_*$ coincide with the classical inner radius of curvature $1/\kappa$, where
	$$\kappa:=\max\{0,\kappa_1,\dots,\kappa_{n-1}\}.$$
	Here $\kappa_1,\dots,\kappa_{n-1}$ are the inner principal curvatures of $\partial\Omega$ and we interpret $1/\kappa=\infty$ when $\kappa=0$.
\end{remark}

Finally we review a well-known result about the continuity of the metric projection which is frequently used in this paper.
The proof is elementary so is safely omitted.

\begin{lemma}\label{metprolem}
	For any open set $\Omega\subset\mathbb{R}^n$ with nonempty boundary, the map $\pi:\Omega\to\mathcal{P}(\partial\Omega)$ is set-valued upper semicontinuous.
	In particular, the induced map $\hat{\pi}:\Omega\setminus\Sigma\to\partial\Omega$ defined by $x\mapsto\xi\in\pi(x)$ is well-defined and continuous. 
\end{lemma}

\section{Characterization of the cut locus by radius of curvature}

In this section we prove our main theorem (Theorem \ref{thm1}).
Let $e_n$ denote the $n$-th unit vector of $\mathbb{R}^n$ and $\langle\cdot,\cdot\rangle$ denote the Euclidean inner product.
Let $\widetilde{\rm pr}:\mathbb{R}^n\to\mathbb{R}^{n-1}$ be a map induced by the orthogonal projection to the hyperplane $\{\langle y,e_n \rangle =0\}$, that is,
$$\widetilde{\rm pr}(y_1,\dots,y_n):=(y_1,\dots,y_{n-1}).$$
Throughout this section we fix these notations.

\subsection{Upper bound for radius of curvature}

In this subsection we prove the ``if'' part of Theorem $\ref{thm1}$ by proving its contrapositive.
Recalling the definition of $\Sigma$, it suffices to prove the following:

\begin{proposition}\label{mainprop1}
	Let $x\in\overline{\Sigma}\setminus\Sigma$, $\pi(x)=\{\xi\}$ and $\partial\Omega$ be of class $C^1$ near $\xi$.
	Then the inequality $d(x)\geq\rho_*(\xi)$ holds.
\end{proposition}

To prove this proposition we prepare the following two lemmas about locally touching spheres and circles for subgraphs of functions.

Here we say that a continuous function $f_1$ on $A_1\in\mathcal{U}^m_y$ touches a function $f_2$ on $A_2\in\mathcal{U}^m_y$ from below (resp.\ above) at $y\in\mathbb{R}^m$ if the function $f_2-f_1$ defined on $A_1\cap A_2$ attains its local minimum (resp. maximum) at $y$ and $f_1(y)=f_2(y)$ holds.
Moreover, we say that a function defined on an open set $A\subset\mathbb{R}^m$ is an upper semi-sphere function if there are $c_0>0$, $r_0>0$ and $x_0\in\mathbb{R}^m$ such that $$f(x)=\sqrt{r_0^2-|x-x_0|^2}+c_0$$
holds for all $x\in A$.
When $m=1$ it is called an upper semi-circle function.

\begin{lemma}\label{spherelem1}
	Let $h:[a,b]\to\mathbb{R}$ be a continuous function and $\tilde{h}:[a,b]\to\mathbb{R}$ be a part of an upper semi-circle function with radius $\tilde{r}>(b-a)/2$ such that $h\geq\tilde{h}$ in $[a,b]$ and $h=\tilde{h}$ on $\partial[a,b]$.
	Then there exists $c\in(a,b)$ such that any upper semi-circle function with radius larger than $\tilde{r}$ can not touch $h$ from below at $c$.
\end{lemma} 

\begin{proof}
	The case $h\equiv\tilde{h}$ is obvious thus we assume $h\not\equiv\tilde{h}$.
	Since $h\geq\tilde{h}$ and $h-\tilde{h}$ is continuous on $[a,b]$ we can take a constant $\alpha:=\max(h-\tilde{h})>0$.
	Then $\tilde{h}+\alpha\geq h$ holds in $[a,b]$ and, noting the boundary condition, there exists $c\in(a,b)$ such that $\tilde{h}(c)+\alpha=h(c)$.
	In particular, the upper semi-circle function $\tilde{h}+\alpha$ with radius $\tilde{r}$ touches $h$ from above at $c$ (see Figure \ref{spherefig1}).
	This implies that any upper semi-circle function with radius larger than $\tilde{r}$ can not touch $h$ from below at $c$.
\end{proof}

\begin{figure}[tb]
	\begin{center}
		\includegraphics[width=60mm]{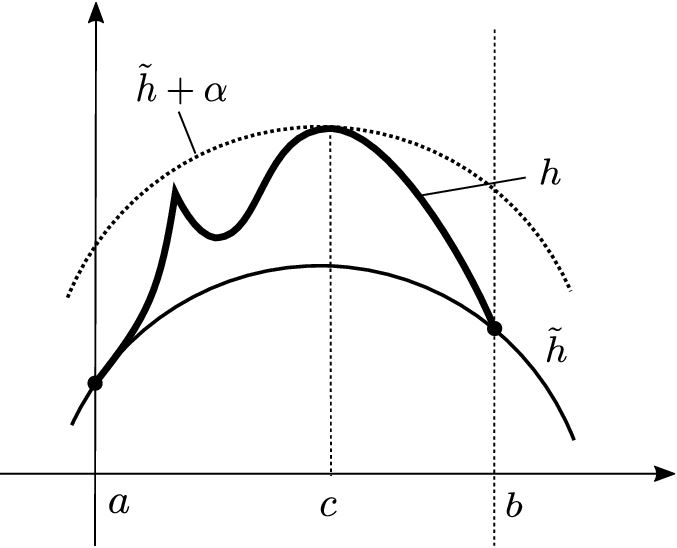}
		\caption{Continuous function on upper semi-circle}
		\label{spherefig1}
	\end{center}
\end{figure}

\begin{remark}\label{sphererem1}
	In the above lemma, the fact is that the subgraph $\{y<h(x)\}$ does not admit a locally inner touching ball (disk) with radius larger than $\tilde{r}$ at $(c,h(c))$.
	We should be careful that, in general, the existence of a locally inner touching ball for the subgraph of a function $f$ does not imply the existence of an upper semi-sphere function with the same radius touching $f$ from below (the converse is generally valid).
	A counterexample is given by $f(x)=\textrm{sign}(x)\sqrt{1-(|x|-1)^2}$.
	In this case any upper semi-circle function does not touch $f$ from below at $0$ since $\lim_{x\to\pm0}f'(x)=\infty$ but the subgraph admits a locally inner touching ball with radius $1$ at the origin.
	However, in particular, if $f$ is differentiable at a point then the two conditions are equivalent there.
	(Actually, it is sufficient that $f$ is touched by a cone function from above.)
\end{remark}

\begin{lemma}\label{spherelem2}
	Let $m\geq1$, $x_0\in\mathbb{R}^m$, $U\in\mathcal{U}^m_{x_0}$ and $f:U\to \mathbb{R}$ be a continuous function which is differentiable at $x_0$.
	Suppose that there exists an upper semi-sphere function with radius $\tilde{r}$ touching $f$ from below at $x_0$.
	Then for any open segment $I\subset U$ containing $x_0$ there exists an upper semi-circle function with radius not smaller than $\frac{\tilde{r}}{\sqrt{1+|\nabla f(x_0)|^2}}$ touching $f|_I$ from below at $x_0$.
\end{lemma}

\begin{proof}
	Without loss of generality, we may assume that $|x_0|<\tilde{r}$, $I\subset B^m_{\tilde{r}}(0)$ and the upper semi-sphere in the assumption is represented by $$\tilde{f}(x)=\sqrt{\tilde{r}^2-|x|^2}.$$
	Since $\tilde{f}$ touches $f$ from below at $x_0$, we have $f\geq\tilde{f}$ near $x_0$ and
	$$\nabla f(x_0)=\nabla\tilde{f}(x_0)=\frac{-x_0}{\sqrt{\tilde{r}^2-|x_0|^2}}.$$
	Now let $L\subset\mathbb{R}^n$ be the line through $I$ and $x_*\in L$ be a unique point so that $|x_*|=\min_{x\in L}|x|$.
	Since $x_0\in I\subset L$, notice that $|x_*|\leq|x_0|$.
	Moreover for any $x\in I$, noting that $x-x_*$ is perpendicular to $x_*$, we have
	$$\tilde{f}|_I(x)=\sqrt{\tilde{r}^2-|x_*|^2-|x-x_*|^2}$$
	which implies that $\tilde{f}|_I$ is an upper semi-circle function with radius $\sqrt{\tilde{r}^2-|x_*|^2}$.
	Therefore, noting that $\tilde{f}|_I$ touches $f|_I$ from below at $x_0\in I$ and
	$$\sqrt{\tilde{r}^2-|x_*|^2}\geq\sqrt{\tilde{r}^2-|x_0|^2}=\frac{\tilde{r}}{\sqrt{1+|\nabla f(x_0)|^2}},$$
	we obtain the conclusion.
\end{proof}

Now we prove Proposition \ref{mainprop1}.
We could refer Figure \ref{spherefig2} in the proof.
Without loss of generality, we may assume $x=0$ and $\xi=de_n$, where $d=d(x)>0$.
Then we can represent $\partial\Omega$ by a $C^1$-graph near $\xi$ in the direction of $e_n$: there exist $0<\varepsilon_0<d$, a neighborhood $U_\xi\in\mathcal{U}^n_\xi$ and a $C^1$-function $g$ on $B^{n-1}_{\varepsilon_0}(0)$ such that $g(0)=d$, $\nabla g(0)=0$ and $\tilde{g}(B^{n-1}_{\varepsilon_0}(0))=\partial\Omega\cap U_\xi$, where $\tilde{g}(\cdot):=(\cdot,g(\cdot))$.

\begin{proof}[Proof of Proposition \ref{mainprop1}]
	Since $x=0\in\overline{\Sigma}$, there exists a sequence $\{x_k\}\subset\Sigma$ such that $x_k\to 0$.
	Then for any $k$ the set $\pi(x_k)\subset\partial\Omega$ has at least two elements.
	We denote them by $\xi_k^1$ and $\xi_k^2$.
	Since $\pi$ is set-valued upper semicontinuous by Lemma \ref{metprolem}, both $\xi_k^1$ and $\xi_k^2$ converge to $\xi$ as $k\to\infty$.
	Thus we may assume that for any $k$
	$$\xi_k^1,\xi_k^2\in\tilde{g}(B^{n-1}_{\varepsilon_0}(0))=\partial\Omega\cap U_\xi$$
	and the $n$-th components of $\xi_k^1-x_k$ and $\xi_k^2-x_k$ are positive, that is, the points $\xi_k^1$ and $\xi_k^2$ lie in the upper semi-sphere part of $S^{n-1}_{d(x_k)}(x_k)$.
	Define $\xi_k^{'i}:=\widetilde{\rm pr}(\xi_k^i)\in B^{n-1}_{\varepsilon_0}(0)$ for $i=1,2$.
	Noting that $\xi_k^{'1}\neq\xi_k^{'2}$, we can define $I_k$ as the closed segment joining $\xi_k^{'1}$ to $\xi_k^{'2}$.
	
	Since the radius of any circle obtained as a section of $S^{n-1}_{d(x_k)}(x_k)$ is at most $d(x_k)$, the function $g|_{I_k}$ satisfies the assumption of Lemma \ref{spherelem1} for an upper semi-circle function with radius not larger than $d(x_k)$.
	Thus, by Lemma \ref{spherelem1}, there exists $\xi_k^{'3}\in I_k\setminus\{\xi_k^{'1},\xi_k^{'2}\}$ such that any upper semi-circle function with radius larger than $d(x_k)$ can not touch $g|_{I_k}$ from below at $\xi_k^{'3}$.
	
	Therefore, by the contrapositive of Lemma \ref{spherelem2}, any upper semi-sphere function with radius larger than $d(x_k)\sqrt{1+|\nabla g(\xi_k^{'3})|^2}$ can not touch $g$ from below at $\xi_k^{'3}$.
	Noting Remark \ref{sphererem1} and that $g$ is differentiable, the above fact yields the nonexistence of locally inner touching balls with radius larger than $d(x_k)\sqrt{1+|\nabla g(\xi_k^{'3})|^2}$ at $\xi_k^3:=\tilde{g}(\xi_k^{'3})\in\partial\Omega$, thus we have
	$$\rho(\xi_k^3)\leq d(x_k)\sqrt{1+|\nabla g(\xi_k^{'3})|^2}.$$
	Since $d(x_k)\to d(x)$, $\xi_k^{'3}\to0$ and $\xi_k^3\to\xi$ as $k\to\infty$ and $g$ is of class $C^1$, we obtain
	$$\rho_*(\xi)\leq\liminf_{k\to\infty}\rho(\xi_k^3)\leq d(x)\sqrt{1+|\nabla g(0)|^2}=d(x)$$
	which is the desired inequality.
\end{proof}

\begin{figure}[tb]
	\begin{center}
		\includegraphics[width=100mm]{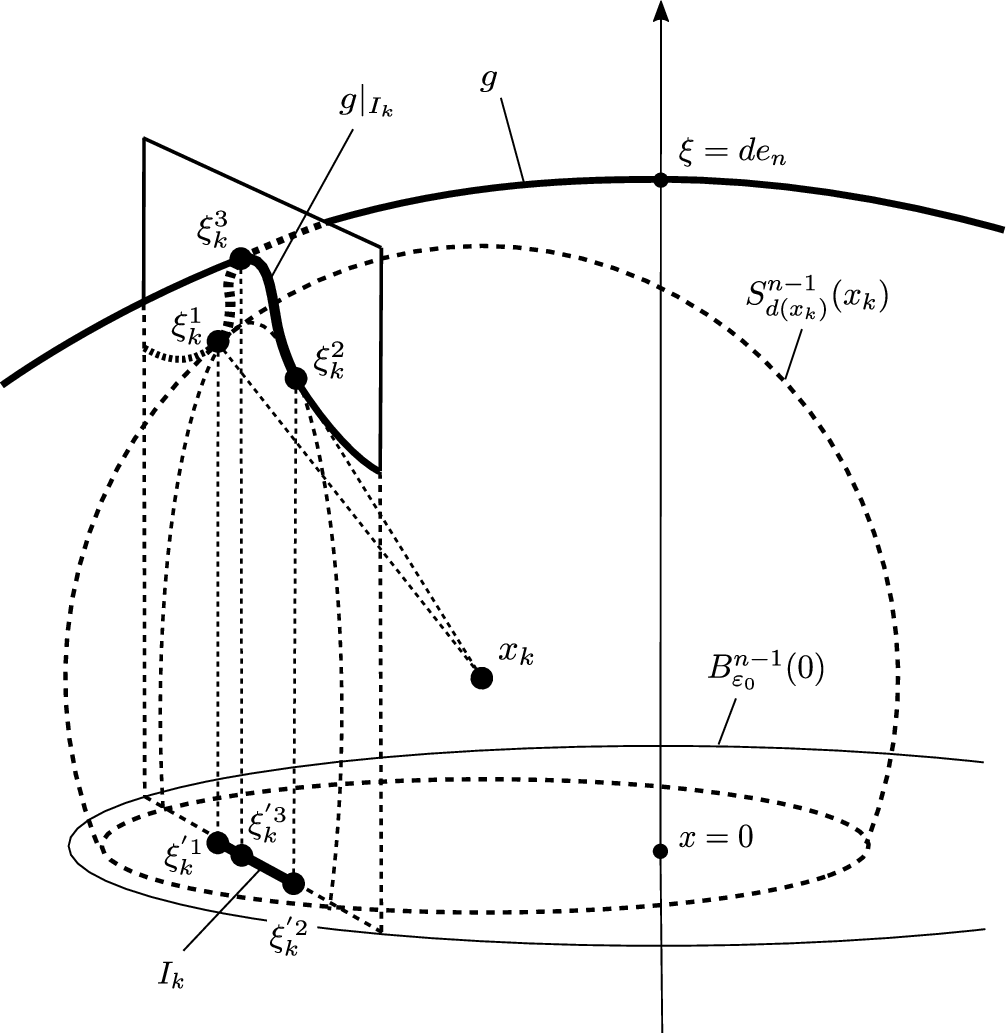}
		\caption{Panorama of the sequences: proof of Proposition \ref{mainprop1}}
		\label{spherefig2}
	\end{center}
\end{figure}

\subsection{Lower bound for radius of curvature}

In this subsection we prove the ``only if'' part of Theorem $\ref{thm1}$ for a class of $\Omega\subset\mathbb{R}^n$.
The class contains all open sets with $C^1$-boundary.
We first define this class.

\begin{definition}\label{defnonsprper}
	Let $\Omega\subset\mathbb{R}^n$ be an open set with nonempty boundary.
	We say that a point $\eta\in\partial\Omega$ has {\it non-spreading inner perpendicular} if either, $\pi^{-1}(\{\eta\})$ is empty, or nonempty $\pi^{-1}(\{\eta\})$ is contained in a (unique) line $L\subset\mathbb{R}^n$.
	In addition, if $\partial\Omega$ is locally represented by a graph near $\eta$ in the direction of $L$, then we say that the point $\eta\in\partial\Omega$ has {\it non-spreading inner perpendicular with graph representation}.
\end{definition}

\begin{remark}
	Note that if $\eta\in\partial\Omega$ has non-spreading inner perpendicular and there exists $y\in\Omega$ such that $\eta\in\pi(y)$ then, even when $y\not\in\pi^{-1}(\{\eta\})$, the pull-back $\pi^{-1}(\{\eta\})$ is not empty and the line $L$ passes through $y$ and $\eta$.
	That is because for any $t\in(0,1)$ the point $ty+(1-t)\eta$ belongs to $\pi^{-1}(\{\eta\})$.
\end{remark}

\begin{remark}	
	For any open set with $C^1$-boundary, all points in its boundary have non-spreading inner perpendicular with graph representation (in this case $L$ is in the normal direction of $\partial\Omega$).
	In addition, the same property is valid for some more non-smooth sets, for example general convex open sets.
\end{remark}

Then the ``only if'' part follows by proving the following.

\begin{proposition}\label{mainprop2}
	Let $x\in\Omega\setminus\overline{\Sigma}$ and $\pi(x)=\{\xi\}$.
	Suppose that $\xi\in\partial\Omega$ has non-spreading inner perpendicular with graph representation.
	Then there exist $\delta>0$ and $r_\delta>0$ such that for any $\eta\in B^n_{r_\delta}(\xi)\cap\partial\Omega$ the inequality $d(x)+\delta\leq\rho(\eta)$ holds.
\end{proposition}

To prove this, we should prepare a key lemma (Lemma \ref{homlem3}) about the image of the metric projection $\hat{\pi}$.
To this end we check a basic property of continuous maps.

\begin{lemma}\label{homlem2}
	Let $m\geq1$, $D\subset\mathbb{R}^m$ be a bounded region containing the origin and $f:\overline{D}\to\mathbb{R}^m$ be a continuous map such that
	\begin{align}\label{innpro}
		\langle f(y),y \rangle + |f(y)||y| >0 \quad \text{for any}\quad y\in \partial D.
	\end{align}
	Then there exists $\bar{r}>0$ such that $B^m_{\bar{r}}(0)\subset f(D)$.
\end{lemma}

\begin{proof}
	Define the homotopy between the identity map $\textrm{id}$ and the function $f$ by $$H(y,t):=tf(y)+(1-t)y.$$
	Then, using (\ref{innpro}), it is easily confirmed that $|H(y,t)|^2>0$ for any $(y,t)$ in $\partial D\times[0,1]$.
	Since $\partial D\times[0,1]$ is compact and $H$ is continuous, there exists $\bar{r}>0$ such that $|H|>\bar{r}$ on $\partial D\times[0,1]$.
	Then for any $z\in B^m_{\bar{r}}(0)$ we have $z\not\in H(\partial D,[0,1])$ and hence
	$$\deg(f,D,z)=\deg(\textrm{id},D,z)=1$$
	by the homotopy invariance of mapping degree (see e.g.\ \cite[\S IV.2]{OuRu09}).
	This implies that there exists $x\in D$ such that $f(x)=z$.
	Thus $B^m_{\bar{r}}(0)\subset f(D)$.
\end{proof}

Now let us go back in our situation (Proposition \ref{mainprop2}).
We may assume again $x=0$ and $\xi=de_n$, where $d=d(x)>0$.
Under the assumption of Proposition \ref{mainprop2}, $\partial\Omega$ has a local graph representation near $\xi=de_n$ in the direction of $e_n$: there exist $\varepsilon_0>0$, a neighborhood $U_\xi\in\mathcal{U}^n_\xi$ and a function $g$ on $B^{n-1}_{\varepsilon_0}(0)$ such that $g(0)=d$ and $\tilde{g}(B^{n-1}_{\varepsilon_0}(0))=\partial\Omega\cap U_\xi$, where $\tilde{g}(\cdot):=(\cdot,g(\cdot))$.

Since $\Omega\setminus\overline{\Sigma}$ is open, we may also assume $B^n_{\varepsilon_0}(0)\subset\Omega\setminus\overline{\Sigma}$.
Moreover, we define a lower semi-sphere in $B^n_{\varepsilon_0}(0)$ by
\begin{align}\label{semisphere}
	S^{n-1}_{\varepsilon_1,-}:=\{y\in B^n_{\varepsilon_0}(0) \mid \langle y,e_n \rangle \leq 0,\ |y|=\varepsilon_1 \}
\end{align}
so that $\hat{\pi}(S^{n-1}_{\varepsilon_1,-})\subset\partial\Omega\cap U_\xi$.
Such $0<\varepsilon_1<\varepsilon_0$ can be taken since $\hat{\pi}$ is continuous by Lemma \ref{metprolem}.

Then the following key lemma holds (see also Figure \ref{imagefig1}).

\begin{lemma}\label{homlem3}
	Suppose the assumption of Proposition \ref{mainprop2} and let $S^{n-1}_{\varepsilon_1,-}$ be the semi-sphere defined as (\ref{semisphere}).
	Then there exists $\bar{r}>0$ such that
	$$B^n_{\bar{r}}(\xi)\cap\partial\Omega\subset\hat{\pi}(S^{n-1}_{\varepsilon_1,-}).$$
\end{lemma}

\begin{proof}
	Define a homeomorphism $p_1:S^{n-1}_{\varepsilon_1,-}\to\overline{B}^{n-1}_{\varepsilon_1}(0)$ by the restriction of $\widetilde{\rm pr}$, namely $p_1:=\widetilde{\rm pr}|_{S^{n-1}_{\varepsilon_1,-}}$.
	If the map
	$$\widetilde{\rm pr}\circ\hat{\pi}\circ p_1^{-1}:\overline{B}^{n-1}_{\varepsilon_1}(0)\to\mathbb{R}^{n-1}$$
	satisfies the assumption of Lemma \ref{homlem2}, then there exists small $0<\bar{r}<\varepsilon_0$ such that $B^{n}_{\bar{r}}(\xi)\subset U_\xi$ and
	$$B^{n-1}_{\bar{r}}(0)\subset \widetilde{\rm pr}\circ\hat{\pi}\circ p_1^{-1}(\overline{B}^{n-1}_{\varepsilon_1}(0))=\widetilde{\rm pr}\left(\hat{\pi}(S^{n-1}_{\varepsilon_1,-})\right).$$
	Using these inclusions and noting that $\hat{\pi}(S^{n-1}_{\varepsilon_1,-})\subset\partial\Omega\cap U_\xi$, we obtain the conclusion
	$$B^{n}_{\bar{r}}(\xi)\cap\partial\Omega\subset\tilde{g}(B^{n-1}_{\bar{r}}(0))\subset\tilde{g}\left(\widetilde{\rm pr}\left(\hat{\pi}(S^{n-1}_{\varepsilon_1,-})\right)\right)=\hat{\pi}(S^{n-1}_{\varepsilon_1,-}).$$
	
	Therefore, it suffices to confirm that the map $\widetilde{\rm pr}\circ\hat{\pi}\circ p_1^{-1}$ satisfies the assumption of Lemma \ref{homlem2}, where $\overline{D}=\overline{B}^{n-1}_{\varepsilon_1}(0)$.
	This map is obviously continuous by Lemma \ref{metprolem}.
	The condition (\ref{innpro}) is proved as follows.
	Fix any $y'\in\partial\overline{B}^{n-1}_{\varepsilon_1}(0)$ and denote
	$$y:=p_1^{-1}(y')=(y',0)\in\mathbb{R}^{n}.$$
	Clearly, $d(y)\leq|\xi-y|$ holds.
	Moreover, we see that $d(y)<|\xi-y|$ since $\partial\Omega$ has non-spreading inner perpendicular.
	Indeed, the equality $d(y)=|\xi-y|$ implying $\hat{\pi}(y)=\xi$ can not be attained since if so then $\{0,y\}\subset\pi^{-1}(\{\xi\})$ but the three points $0$, $\xi$, $y\in\mathbb{R}^n$ are not in alignment.
	Thus we find that $\hat{\pi}(y)\in B^n_{|\xi-y|}(y)\cap\partial\Omega$.
	In addition, we find that $\hat{\pi}(y)\notin\overline{B}^n_d(0)$ since $\overline{B}^n_d(0)\cap\partial\Omega=\{\xi\}$ and $\xi\notin B^n_{|\xi-y|}(y)$.
	Noting that any $z\in B^n_{|\xi-y|}(y)\setminus\overline{B}^n_d(0)$ satisfies $\langle \widetilde{\rm pr}(z),\widetilde{\rm pr}(y) \rangle>0$ by the shape of $B^n_{|\xi-y|}(y)\setminus\overline{B}^n_d(0)$, we have
	$$\langle \widetilde{\rm pr}\left(\hat{\pi}(y)\right),\widetilde{\rm pr}(y) \rangle = \langle \widetilde{\rm pr}\circ\hat{\pi}\circ p_1^{-1}(y'),y' \rangle >0.$$
	This immediately implies the condition (\ref{innpro}).
\end{proof}

\begin{figure}[tb]
	\begin{center}
		\includegraphics[width=60mm]{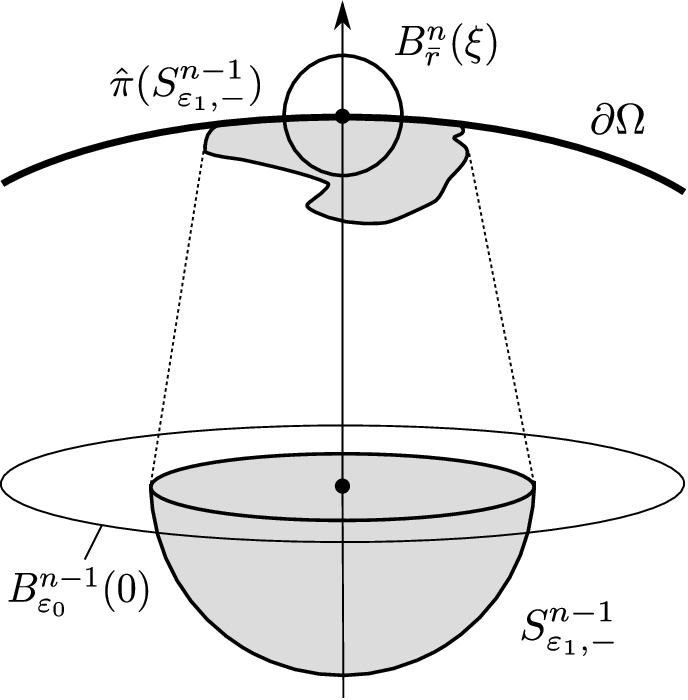}
		\caption{Image of the metric projection $\hat{\pi}(S^{n-1}_{\varepsilon_1,-})$}
		\label{imagefig1}
	\end{center}
\end{figure}

We are now in position to prove the main proposition.

\begin{proof}[Proof of Proposition \ref{mainprop2}]
	By Lemma \ref{homlem3}, there exists $\bar{r}>0$ such that
	$$B^n_{\bar{r}}(\xi)\cap\partial\Omega\subset\hat{\pi}(S^{n-1}_{\varepsilon_1,-}).$$
	Therefore, for any $\eta\in B^n_{\bar{r}}(\xi)\cap\partial\Omega$ there exists $y_\eta\in S^{n-1}_{\varepsilon_1,-}$ such that $\hat{\pi}(y_\eta)=\eta$ which implies that $B^n_{|\eta-y_\eta|}(y_\eta)$ is an inner touching ball at $\eta\in\partial\Omega$.
	This yields that $\rho(\eta)\geq|\eta-y_\eta|$.
	In addition, by the definition of $S^{n-1}_{\varepsilon_1,-}$, there exist $\delta>0$ and $0<r_\delta<\bar{r}$ such that the distance between $B^n_{r_\delta}(\xi)$ and $S^{n-1}_{\varepsilon_1,-}$ is not smaller than $d(x)+\delta$.
	Then for any $\eta\in B^n_{r_\delta}(\xi)\cap\partial\Omega$ we have
	$$\rho(\eta)\geq|\eta-y_\eta|\geq d(x)+\delta.$$
	The proof is completed.
\end{proof}

\begin{remark}
	The above proof does not use the differentiability of the distance function.
	We have another proof of this part, which uses the implicit function theorem for the distance function (and requires slightly different assumptions) but is considerably shorter.
	The statement and proof are given in the rest of this section.
	We would like to thank a referee for suggesting this tack in the reviewing process.
\end{remark}

\begin{proposition}\label{mainprop3}
	Let $x\in\Omega\setminus\overline{\Sigma}$ and $\pi(x)=\{\xi\}$.
	Suppose that there is a neighborhood $U$ of $\xi$ in $\partial\Omega$ such that any $\eta\in\partial\Omega\cap U$ has non-spreading inner perpendicular and $\partial\Omega\cap U$ is an $(n-1)$-dimensional topological manifold.
	Then there exist $\delta>0$ and $r_\delta>0$ such that for any $\eta\in B^n_{r_\delta}(\xi)\cap\partial\Omega$ the inequality $d(x)+\delta\leq\rho(\eta)$ holds.
\end{proposition}

\begin{proof}
	Since $x\in\Omega\setminus\overline{\Sigma}$, there is small $r>0$ such that the function $d$ is of class $C^1$ in $B^n_r(x)\subset\Omega\setminus\overline{\Sigma}$ and, by Lemma \ref{metprolem}, $\hat{\pi}(B^n_r(x))\subset\partial\Omega\cap U$.
	Define $$\Gamma=\{y\in B^n_r(x) \mid d(y)=d(x)\}.$$
	Noting that $|\nabla d|=1$ if it is differentiable, $\Gamma$ is a $C^1$ hypersurface by the implicit function theorem.
	By the assumption of non-spreading inner perpendicular, the restriction map $\hat{\pi}|_\Gamma$ is injective.
	Hence, noting the assumption of topological manifold, the map $\hat{\pi}|_\Gamma$ is a homeomorphism to its image by the invariance of domain theorem; in particular, the image of any neighborhood of $x$ in $\Gamma$ of $\hat{\pi}$ is a neighborhood of $\xi$ in $\partial\Omega$.
	For any small open region neighborhood $V$ of $x$ in $\Gamma$, the boundary of $\hat{\pi}(V)$ in $\partial\Omega$ is strictly far from $\xi$, thus even if the set $V$ is slightly shifted in the direction away from $\partial\Omega$ (while $V\subset B^n_r(x)$) the image of $\hat{\pi}$ still contains a neighborhood of $\xi$.
	This is justified by the theory of mapping degree.
	The proof is now complete.
\end{proof}

\section{A sufficient condition for vanishing of Lebesgue measure}

In this section, for $\Omega\subset\mathbb{R}^n$ with $C^{1,1}$-boundary we give a sufficient condition that the Lebesgue measure of the cut locus $\overline{\Sigma}$ vanishes by utilizing Theorem \ref{thm1}.
Let $\mathcal{L}^n$ and $\mathcal{H}^{n-1}$ denote the $n$-dimensional Lebesgue measure and the $(n-1)$-dimensional Hausdorff measure.

The following is the main theorem in this section.

\begin{theorem}\label{lebthm}
	Let $\Omega\subset\mathbb{R}^n$ be an open set with $C^{1,1}$-boundary.
	If there is an $\mathcal{H}^{n-1}$-negligible subset $\Gamma\subset\partial\Omega$ such that $\rho_*=\rho$ in $\partial\Omega\setminus\Gamma$, then $\mathcal{L}^n(\overline{\Sigma})=0$.
\end{theorem}

By this theorem we immediately obtain:

\begin{corollary}\label{lebcor}
	Let $\Omega\subset\mathbb{R}^n$ be an open set with $C^{1,1}$-boundary.
	If there is an $\mathcal{H}^{n-1}$-negligible subset $\Gamma\subset\partial\Omega$ such that $\partial\Omega\setminus\Gamma$ is $C^2$, then $\mathcal{L}^n(\overline{\Sigma})=0$.
\end{corollary}

Here that $\partial\Omega\setminus\Gamma$ is $C^2$ means that for any $\xi\in\partial\Omega\setminus\Gamma$ there is an open neighborhood $U\in\mathcal{U}^n_\xi$ such that $(\partial\Omega\setminus\Gamma)\cap U$ is a $C^2$-manifold.
Of course the assumptions in Theorem \ref{lebthm} and Corollary \ref{lebcor} exclude the example of Mantegazza-Mennucci \cite[\S3]{MaMe03}.

We now prove Theorem \ref{lebthm}.
Our proof is inspired by \cite{CrMa07}.
We first prepare the following inclusion relation.

\begin{proposition}\label{incprop}
	Let $\partial\Omega$ be of class $C^1$.
	Define for $\xi\in\partial\Omega$ and $0\leq t<\infty$
	$$\Psi(\xi,t):=\xi+t\nu(\xi),$$
	where $\nu(\xi)$ is the inner unit normal at $\xi$, and
	$$\mathcal{S}:=\{\Psi(\xi,t)\in\mathbb{R}^n\mid \xi\in\partial\Omega,\ \rho_*(\xi)<\infty,\ \rho_*(\xi)\leq t\leq \rho(\xi)\}.$$
	Then $\overline{\Sigma}\setminus\Sigma\subset\mathcal{S}$.
\end{proposition}

\begin{proof}
	Fix any $x\in\overline{\Sigma}\setminus\Sigma$ and let $\xi=\hat{\pi}(x)\in\partial\Omega$.
	By Theorem \ref{thm1} we have
	$$\rho_*(\xi)\leq d(x)=|x-\xi|<\infty.$$
	Noting that $B_{|x-\xi|}^n(x)$ is an inner touching ball at $\xi$ of $\Omega$, we also have
	$$|x-\xi|\leq\rho(\xi).$$
	Since $\nu(\xi)=\frac{x-\xi}{|x-\xi|}$, we obtain $\Psi(\xi,|x-\xi|)=x$ and hence $x\in\mathcal{S}$.
\end{proof}

Noting the above proposition and $\mathcal{L}^n(\Sigma)=0$ (see e.g.\ \cite{Al94}), it suffices to confirm the following proposition in order to prove Theorem \ref{lebthm}.

\begin{proposition}
	Let $\partial\Omega$ be of class $C^{1,1}$ and suppose that there is an $\mathcal{H}^{n-1}$-negligible subset $\Gamma\subset\partial\Omega$ such that $\rho_*(\xi)=\rho(\xi)$ holds for any $\xi\in\partial\Omega\setminus\Gamma$.
	Then $\mathcal{L}^n(\mathcal{S})=0$.
\end{proposition}

\begin{proof}
	Under the assumption, by Proposition \ref{incprop} we notice that $\mathcal{S}\subset\mathcal{S}^1\cup\mathcal{S}^2$, where
	$$\mathcal{S}^1:=\{\Psi(\xi,t)\in\mathbb{R}^n\mid \xi\in\partial\Omega,\ t=\rho_*(\xi)<\infty\},$$
	$$\mathcal{S}^2:=\{\Psi(\xi,t)\in\mathbb{R}^n\mid \xi\in\Gamma,\ t\geq0\}.$$
	Thus it suffices to confirm that $\mathcal{L}^n(\mathcal{S}^1)=\mathcal{L}^n(\mathcal{S}^2)=0$.
	
	Since $\partial\Omega$ is $C^{1,1}$, there exist countable local parametrizations: for $k\in\mathbb{N}$ there are open sets $U_k\subset\mathbb{R}^{n-1}$, $V_k\subset\mathbb{R}^n$, a compact set $C_k\subset U_k$ and a $C^{1,1}$-diffeomorphism $Y_k:U_k\to \partial\Omega\cap V_k$ such that $\partial\Omega=\bigcup_{k}Y_k(C_k)$.
	Then, noting that $\nu$ is Lipschitz since $\partial\Omega$ is $C^{1,1}$, we can define Lipschitz maps $\overline{\Psi}_k:\mathbb{R}^n\supset C_k\times[0,\infty)\to\mathbb{R}^n$ by
	$$\overline{\Psi}_k(x',t):=\Psi(Y_k(x'),t)=Y_k(x')+t\nu(Y_k(x')).$$
	
	Now we define a set $A^1_k\subset C_k\times[0,\infty)$ for any $k$ by
	$$A^1_k:=\{(x',t)\in C_k\times[0,\infty)\mid t=\rho_*(Y_k(x'))\}.$$ 
	Since $A^1_k$ is a part of the graph of the lower semicontinuous function $\rho_*\circ Y'_k$, we have $\mathcal{L}^n(A^1_k)=0$.
	Noting that $\mathcal{S}^1=\bigcup_k \overline{\Psi}_k(A^1_k)$, we obtain
	$$\mathcal{L}^n(\mathcal{S}^1)\leq \sum_k\mathcal{L}^n(\overline{\Psi}_k(A^1_k))\leq \sum_k L_k^n\mathcal{L}^n(A^1_k)=0,$$
	where $L_k$ is the Lipschitz constant of $\overline{\Psi}_k$ on $C_k$.
	
	Next we define a set $A^2_k\subset C_k\times[0,\infty)$ for any $k$ by	
	$$A^2_k:=Y_k^{-1}(\Gamma\cap Y_k(C_k))\times[0,\infty).$$
	By the area formula (see e.g.\ \cite{EvGa15}) and $\mathcal{H}^{n-1}(\Gamma)=0$, we have
	$$\mathcal{L}^{n-1}(Y_k^{-1}(\Gamma\cap Y_k(C_k)))\leq J_k^{-1}\mathcal{H}^{n-1}(\Gamma\cap Y_k(C_k))=0,$$
	where $J_k:=\min_{C_k}|JY_k|>0$.
	Thus we have $\mathcal{L}^n(A^2_k)=0$ and hence, similarly as above, $\mathcal{L}^n(\mathcal{S}^2)=0$.
\end{proof}

\section{A counterexample with Lipschitz boundary}

The statement of Theorem \ref{thm1} does not hold for Lipschitz cases.
A counterexample $\Omega\subset\mathbb{R}^2$ is simply given as the union of the unit disc $\{x_1^2+x_2^2<1\}$ and the upper half plane $\{x_2>0\}$.
Note that $\overline{\Sigma}=\{x_1=0\}\cap\{x_2\geq0\}$.
Then, for example, the point $x=(1,2)\in\Omega\setminus\overline{\Sigma}$ satisfies $\hat{\pi}(x)=(1,0)$ but we find that $d(x)=2$ and
$$\rho_*(\hat{\pi}(x))=\lim_{\theta\uparrow0}\rho((\cos\theta,\sin\theta))=1.$$
However, if the boundary has no ``singular concave point'' the statement of Theorem \ref{thm1} can be valid.
This will be discussed in our forthcoming paper.

\subsection*{Acknowledgments}
The author would like to thank Professor Yoshikazu Giga, Mr.\ Hiroyuki Ishiguro and Dr.\ Atsushi Nakayasu for their helpful comments and discussion.
He is also grateful to anonymous referees for useful comments for the improvement of this paper.
This work was supported by a Grant-in-Aid for JSPS Fellows 15J05166 and the Program for Leading Graduate Schools, MEXT, Japan.

\end{document}